\documentclass[11pt,leqno]{amsart}
\usepackage[utf8]{inputenc}
\usepackage{bbm}
\usepackage{eulervm}

\usepackage[T1]{fontenc}
\usepackage[a4paper,asymmetric]{geometry}
\usepackage{latexsym,amssymb}
\usepackage{hyperref}
\usepackage{amsmath,amsthm}
\usepackage{amsfonts}
\usepackage[american]{babel}
\usepackage{mathrsfs}
\usepackage{psfrag}
\usepackage{graphicx}
\usepackage{xcolor}

\newcommand{\R}{\mathbb{R}}
\newcommand{\rn}{\mathbb{R}^N}

\newcommand{\hau}{Hausdorff }

\newcommand{\aaa}{{\textrm{\sf{a}}}}

\newcommand{\bd}{\partial}
\newcommand{\dive}{\text{\sf div}}

\newcommand{\conv}{\text{\sf conv}}

\newcommand{\Deltap}{\Delta_{\p}}

\newcommand{\ep}{\varepsilon}

\newcommand{\F}{\mathscr{F}}
\newcommand{\fm}{\mathscr{F}^-}

\newcommand{\Ha}{\mathscr{H}}

\newcommand{\la}{\lambda}
\newcommand{\La}{\Lambda}
\newcommand{\Lab}[1]{{\boldsymbol\Lambda_{#1}}}

\newcommand{\mink}{Minkowski }

\newcommand{\omk}{\oo\setminus\overline K}
\newcommand{\oo}{\Omega}
\newcommand{\ot}[1]{\Omega(#1)}
\newcommand{\ou}{K}

\newcommand{\p}{\text{p}}

\newcommand{\qm}{\textrm{Q}^2_-}

\newcommand{\ua}{u_0}
\newcommand{\ub}{u_1}
\newcommand{\ula}{u_{\la}}

\newcommand{\urj}{u_{r_j}}

\newcommand{\xx}{\bar{x}}

\numberwithin{equation}{section}
\newtheorem{theorem}{Theorem}[section]
\newtheorem{proposition}[theorem]{Proposition}

\newtheorem{lemma}[theorem]{Lemma}
\newtheorem{remark}[theorem]{Remark}

\begin{document}
\title[Bernoulli problem]{A Bernoulli problem with non constant gradient boundary constraint}

\author[C. Bianchini]{Chiara Bianchini}

\address{C. Bianchini, Institut Elie Cartan, Universit\'e Henri Poincar\'e Nancy, Boulevard des Aiguillettes B.P. 70239, F-54506 Vandoeuvre-les-Nancy Cedex, France}
\email{cbianchini@math.unifi.it}


\date{}


\begin{abstract}
We present in this paper a result about existence and convexity of solutions to a free boundary problem of Bernoulli  type, with non constant gradient boundary constraint depending on the outer unit normal.
In particular we prove that, in the convex case, the existence of a subsolution guarantees the existence of a classical solution, which is proved to be convex.
\end{abstract}

\maketitle

\section{Introduction}
Consider an annular condenser with a constant potential difference equal to one, such that one of the two plates is given and the other one has to be determined in such a way that the intensity of the electrostatic field is constant on it.
If $\omk$ represents the condenser, whose plates are $\oo$ and $K$ (with $\overline K\subseteq\oo$), and $u$ is the electrostatic potential, it holds $\Delta u=0$ in $\omk$ and $|Du|= constant$ on either $\bd\oo$ or $\bd K$, depending on which of them represents the unknown plate. 

This gives rise to the classical Bernoulli problems (interior and exterior), where the involved differential operator is the Laplacian $\Delta$, which expresses the linearity of the electrical conduction law.
However, some physical situations can be better modeled by general power flow laws, then yielding to the $\p$-Laplacian as governing operator. 
Moreover, one can consider the possibility to have a non constant prescribed intensity of the electric field on the free
boundary.
In particular, as the intensity of the electrostatic field $\overrightarrow{E}$ on an equipotential surface is related to its outer unit normal vector, through the curvature of that surface, one can assume $|\overrightarrow{E}|$ to depend on the outer unit normal vector $\nu(x)$ of the unknown boundary. 
In view of these considerations, we deal here with the following problem.

Given a domain in $\oo\subseteq\rn$, a real number $\p>1$ and a smooth function $g:S^{N-1}\to \R$ such that 
\begin{equation}\label{gcC}
c\le g(\textsf{v})\le C\quad\text{ for every }\textsf{v}\in S^{N-1},
\end{equation}
for some $C>c>0$, find a function $u$ and a domain $K$, contained in $\oo$, such that
\begin{equation}\label{Bint}
\begin{cases}
\Deltap u(x)=0 \quad &\text{in }\oo\setminus\overline\ou,\\
u=0 \quad &\text{on }\bd\oo,\\
u=1, \quad &\text{on }\bd\ou,\\
|Du(x)|=g(\nu(x)), \quad &\text{on }\bd\ou,
\end{cases}
\end{equation}
where $\nu(x)=\nu_K(x)$ is the outer unit normal to $\bd K$ at $x\in\bd K$.

Here an later $\Deltap$ is the $\p$-Laplace operator for $\p>1$, that is
$$
\Deltap u=\dive(|Du|^{\p-2}Du)\,.
$$

If $u$ is a solution to (\ref{Bint}) we will tacitly continue $u$ by $1$ in $K$ throughout the paper, so that a solution $u$ to (\ref{Bint}) is defined, and continuous, in the whole $\oo$.

The boundary condition $|Du|=\tau$ has to be understood in a classical way:
$$
\lim_{\substack{y\to x \\ y\in \omk}}|Du(y)|=|Du(x)|.
$$
Moreover, in the convex case, that is when $\oo$ is a convex set, we are allowed to consider classical solutions (justified by \cite{L}, since $K$ inherits the convexity of $\oo$, as shown later).

Notice that, given $ K$ in (\ref{Bint}), the function $u$ is uniquely determined since it represents the capacitary potential of $\omk$; on the other hand, given the function $u$ the free boundary $\bd K$ is determined as $\bd K=\bd\{x\in\rn\ :\ u(x)\ge 1\}$.
Hence, we will speak of {\em a solution} to (\ref{Bint}) referring indifferently to the sets $K$ or to the corresponding potential function $u$ (or to both) and we will indicate the class of solutions as $\F(\oo,g)$, where $\oo$ is the given domain and $g$ is the gradient boundary datum.

The original interior Bernoulli problem corresponds to the case $\p=2$, that is the Laplace operator, with constant gradient boundary constraint $g(\nu(x))\equiv \tau$.
In general, given a domain $\oo\subseteq \rn$, and $\tau>0$, the classical interior Bernoulli problem consists in finding a domain $K$, with $\overline K \subseteq \oo$ and a function $u$ such that
\begin{equation}\label{Bintconstant}
\begin{cases}
\Deltap u(x)=0 \quad &\text{in }\oo\setminus\overline\ou,\\
u=0 \quad &\text{on }\bd\oo,\\
u=1, \quad &\text{on }\bd\ou,\\
|Du|=\tau, \quad &\text{on }\bd\ou.
\end{cases}
\end{equation}
Easy examples show that Problem (\ref{Bintconstant}), and hence Problem (\ref{Bint}), need not have a solution for every given domain $\oo$ and for every positive constant $\tau$. 
Many authors consider the classical problem, both from the side of the existence and geometric properties of the solution. 
In particular we recall the pioneering work of Beurling \cite{B} and several other contributions as  \cite{A2}, \cite{AC}, \cite{F}, \cite{FR}, \cite{La}. 
The treatment of the nonlinear case is more recent and mainly due to Henrot and Shahgholian (see for instance \cite{HS1}, \cite{HS2}; see also \cite{AM},\cite{BS}, \cite{GK}, \cite{MPS} and references therein).
The uniqueness problem has been solved later in \cite{CT} for $\p=2$ and \cite{BS} for $\p>1$.
Here we summarize some of the known results.
%
%

\begin{equation}\label{Bresume}
  \begin{minipage}{0.85\linewidth}
Let $\oo\subseteq\rn$ be a convex $C^1$ bounded domain.
There exists a positive constant $\Lab{\p}=\Lab{\p}(\oo)$, named \emph{Bernoulli constant}, such that Problem (\ref{Bintconstant}) has a solution if and only if $\tau\ge\Lab{\p}$; in such a case there is at least one which is $C^{2,\alpha}$ and convex.
In particular for $\tau=\Lab{\p}$ the solution is unique.
  \end{minipage}
\end{equation}

In this paper we consider Problem (\ref{Bint}) in the convex case, that is when the given domain is a convex set, and we prove that the convexity is inherited by the unknown domain without making additional assumptions on the function $g$.
More precisely, let us indicate by $\fm(\oo,g(\nu))$ the class of the so called \emph{subsolution} to Problem (\ref{Bint}); essentially, $v$ and $K$ are subsolutions if $v$ solves
$$
\begin{cases}
\Deltap v\ge 0       &\qquad\text{ in }\omk\\
v=0                  &\qquad\text{ on }\bd\oo\\
v=1,\ |Dv(x)|\le g(\nu(x))   &\qquad\text{ on }\bd K\,;
\end{cases}
$$
(see Section \ref{secsubsol} for more details).

Our main theorem is the following.
\begin{theorem}\label{gsub-sol}
Let $\oo\subseteq\rn$ be a convex $C^1$ domain, and $g:S^{N-1}\to\R$ be a continuous function such that (\ref{gcC}) holds.
If $\fm(\oo,g(\nu))$ is non empty, then there exists a $C^1$ convex domain $K$ with $\overline K\subseteq\oo$ such that the $\p$-capacitary potential $u$ of $\omk$ is a classical solution to the interior Bernoulli problem (\ref{Bint}).
\end{theorem}

The idea of a non constant boundary gradient condition has been developed in the literature by many authors, who considered the case of a space variable dependent constraint, $\aaa:\oo\to(0,+\infty)$.
We refer to \cite{AC},\cite{ACF},\cite{MW} for a functional approach, and to \cite{A2}, \cite{A3}, \cite{HS4} for the subsolution method.
In particular, an analogous result to Theorem \ref{gsub-sol} has been proved in \cite{HS4} where the authors considered 
a Bernoulli problem with non constant gradient boundary datum $\aaa(x)$.
For a given convex domain $\oo\subseteq \rn$, and a positive function $\aaa:\oo\to(0,\infty)$, such that 
$$
c\le\aaa(x)\le C, \text{ for every } x\in \oo,
$$
for some $C>c>0$, with
\begin{equation}\label{acvx}
\frac 1{\aaa} \text{ convex in }\oo,
\end{equation}
they consider the problem
\begin{equation}\label{Bintaaa}
\begin{cases}
\Deltap u(x)=0 \quad &\text{in }\oo\setminus\overline\ou,\\
u=0 \quad &\text{on }\bd\oo,\\
u=1, \quad &\text{on }\bd\ou,\\
|Du(x)|=\aaa(x) \quad &\text{on }\bd\ou.
\end{cases}
\end{equation}
and they proved that, if a subsolution to the problem exists, then there exists a classical solution and moreover the convexity of the given domain transfers to the free boundary.

Notice that in Problem (\ref{Bintaaa}) the function $\aaa$ is required to be given in the whole $\oo$, while in (\ref{Bint}) the function $g$ is defined only on the unit sphere $S^{N-1}$.
Moreover, while in Problem (\ref{Bintaaa}) the convexity property (\ref{acvx}) is required for the boundary constraint $\aaa$, in Problem (\ref{Bint}) no additional assumptions on $g$ are needed.

%

\section{Preliminaries}
\subsection{Notations}
In the $N$-dimensional Euclidean space, $N\geq 2$, we denote by $|\cdot|$ the Euclidean norm; for $K\subseteq\rn$, we denote by $\overline K$ its closure and by $\bd K$ its boundary, while $\conv(K)$ is its convex hull.  
$\Ha^m$ indicates the $m$-dimensional Hausdorff measure. 
We denote by $B(x_0,r)$ the ball in $\rn$ of center $x_0$ and radius $r>0$: $B(x_0,r)=\{x\in\rn\,:\,|x-x_0|<r \}$; in particular $B$ denotes the unit ball $B(0,1)$ and we set $\omega_N=\Ha^N(B)$. 
Let us define
$$
S^{N-1}=\bd B=\{x\in\rn\,:\,|x|= 1\};
$$ 
hence $\Ha^{N-1}(S^{N-1})=N\omega_N$.

We set
$$
\La_m=\{ \la=(\la_1,...,\la_m)\ |\ \la_i\ge0, \sum_{i=1}^m\la_i=1 \}.
$$

Given an open set $\Omega\subseteq \rn$, and a function $u$ of class $C^2(\Omega)$, $Du=(u_{x_1},\dots,u_{x_N})$ and $D^2u=(u_{x_ix_j})_{i,j=1}^N$ denote its gradient and its Hessian matrix respectively.

\subsection{Quasi-concave and $Q^2_-$ functions}
An upper semicontinuous function $u:\rn\to\R\cup\{\pm\infty\}$ is said \emph{quasi-concave} if it has convex superlevel sets, or, equivalently, if
$$
u\left( (1-\la)x_0+\la x_1 \right)\ge \min\{ u(x_0),u(x_1) \},
$$
for every $\la\in[0,1]$, and every $x_0, x_1\in\rn$.
If $u$ is defined only in a proper subset $\Omega$ of $\R^n$, we extend $u$ as $-\infty$ in $\R^n\setminus\Omega$ and we say that $u$ is
quasi-concave in $\Omega$ if such an extension is quasi-concave in $\rn$.
Obviously, if $u$ is concave then it is quasi-concave.

By definition a quasi-concave function determines a family of monotone decreasing convex sets; on the other hand, a continuous family of monotone decreasing convex sets, whose boundaries completely cover the first element, can be seen as the family of super-level sets of a quasi-concave function.

We use a local strengthened version of quasi-concavity, which was introduced and studied in \cite{LS}:
let $u$ be a function defined in an open set $\Omega\subset\R^n$; we say that  $u$ is a $Q^2_-$ function at a point $x\in \Omega$ (and we write $u\in Q^2_-(x)$) if:
\begin{enumerate}
\item $u$ is of class $C^2$ in a neighborhood of $x$;
\item its gradient does not vanish at $x$;
\item the principal curvatures of $\{ y\in\R^n\ |\ u(y)= u(x)\}$ with respect to the normal $-\frac{Du(x)}{|Du(x)|}$ are positive at $x$.
\end{enumerate}

In other words, a $C^2$ function $u$ is $Q^2_-$ at a regular point $\xx$ if its level set $\{x\,:\,u(x)=u(\xx)\}$ is a regular convex surface (oriented according to $-Du$), whose Gauss curvature does not vanish in a neighborhood of $\xx$.
By $u\in Q^2_-(\Omega)$ we mean $u\in Q^2_-(x)$ for every $x\in\Omega$.

\subsection{Quasi concave envelope}
If $u$ is an upper semicontinuous function, we denote by $u^*$ its \emph{quasi-concave envelope}. 
Roughly speaking, $u^*$ is the function whose superlevel sets are the closed convex hulls of the corresponding superlevel sets of $u$. 
It turns out that $u^*$ is also upper semicontinuos.

Let us indicate by $\ot{t}$ the superlevel set of $u$ of value $t$, i.e.
$$
\ot{t}=\{ x\in\rn\ |\ u(x)\ge t \},
$$
and let $\Omega^*(t)=\overline{\conv(\ot{t})}$. 
Then $u^*$ is the function defined by its superlevel sets in the following way:
$$
\Omega^*(t)=\{ x\in\rn\ |\ u^*(x)\ge t \} \qquad \text{ for every }t\in\R\,,
$$
that is
$$
u^*(x)=\sup \{ t\in\R\, |\, x\in\Omega^*(t) \}.
$$
Equivalently, as shown in \cite{CS}, 
\begin{eqnarray*}
 u^*(x)=\max\left\{ \min\{ u(x_1),...,u(x_{N+1}) \} \ :\ x_i\in\overline{\omk}, \exists \la\in\La_{N+1}, \ x=\sum_{i=1}^{N+1}\la_ix_i \right\}.
\end{eqnarray*}

Notice that $u^*$ is the smallest upper semicontinuous quasi-concave function greater than $u$, hence in particular $u^*\ge u$.
Moreover, if $u$ satisfies $\Deltap u=0$ in a convex ring $\omk$ (that is $\oo,K$ convex with $\overline{K}\subseteq\oo$), then it holds $\Deltap u^*\ge 0$ in $\omk$ in the viscosity sense (see for instance \cite{CS}).
\subsection{Subsolutions}\label{secsubsol}
In his pioneering work \cite{B}, Beurling introduced the notion of sub-solution for the classical Problem (\ref{Bintconstant}). 
This concept was further developed by Acker \cite{A2} and then generalized by Henrot and Shahgholian \cite{HS2}, \cite{HS4} to the case $\p>1$, both for constant and for non constant gradient boundary constraint.

Following the same idea, let us introduce the class of {sub-solutions} to the generalized Bernoulli Problem (\ref{Bint}).
Let $\oo$ be a subset of $\rn$; $\fm(\oo,g)$ is the class of functions $v$ that are Lipschitz continuous on $\overline{\oo}$ and such that
\begin{equation}\label{Asubsol}
\begin{cases}
\Deltap v \ge 0 &\qquad\text{in } \{ v<1\}\cap \oo\\
v =0           &\qquad\text{on }\bd \oo\\
|Dv(x)| \le g(\nu(x))  &\qquad\text{on }\bd\{v<1 \}\cap \oo\,.
\end{cases}
\end{equation}
If $v\in\fm(\Omega,g)$ we call it a {\em subsolution}.

As in the definition of solutions, we say that a set $K$ is a {\em subsolution}, and we possibly write $K\in\fm(\Omega,\tau)$ or $(v,K)\in\fm(\Omega,\tau)$, if $K=\{x\in\oo\,:\,v(x)\ge 1\}$ for some $v\in\fm(\oo,\tau)$.

In the standard case $g\equiv\tau$, for some positive constant $\tau$, it is known that the class of subsolutions and that of solutions are equivalent, indeed in \cite{HS2} is proved that, if $\oo$ is a $C^1$ convex domain, and $\fm(\oo,\tau)$ is not empty, then there exists a classical solution to (\ref{Bintconstant}).
In particular it is proved that 
$$
\widetilde{K}(\oo,\tau)= \bigcup_{C\in\fm(\oo,\tau)}C, \qquad \widetilde{u}=\sup_{v\in\fm(\oo,\tau)}v,
$$
solve Problem (\ref{Bintconstant}) and hence, recalling (\ref{Bresume}), it follows as a trivial consequence:
$$
\Lab{\p}(\oo)=\inf\{\tau\ :\ \F(\oo,\tau)\neq\emptyset\}=\inf\{\tau\ :\ \fm(\oo,\tau)\neq\emptyset\}.
$$ 

Regarding the proof of Theorem \ref{gsub-sol}, it is clear that an analogous relation between subsolutions and solutions hold true also in the non constant case, that is:
$$
\widetilde{K}(\oo,g)= \bigcup_{C\in\fm(\oo,g)}C, \qquad \widetilde{u}=\sup_{v\in\fm(\oo,g)}v,
$$
solve Problem (\ref{Bint}) and they are said \emph{maximal solution} to (\ref{Bint}).

\section{Proof of the main result}
In order to give a proof of Theorem \ref{gsub-sol}, some preliminary steps are needed; they are collected in the following propositions and lemmas.

\begin{proposition}\label{gcap}\label{gconvU}
Let $\oo$ be a regular $C^1$ convex subset of $\rn$; let $\ua,\ub\in\fm(\oo,g)$ with $K_0=\{\ua=1\}$ and $K_1=\{\ub=1\}$. 
Define $K=K_0\cup K_1$, $K^*=\overline{\conv{K}}$.
Then $v\in\fm(\oo,g)$, where $v$ is the $\p$-capacitary potential of $\oo\setminus K^*$. 	

Moreover 
$$
|Dv(x)| \le g(\nu_{\ot{t}}(y_x)),
$$	
for every $x\in\oo\setminus K^*$ and $y_x\in\bd K^*$ such that $\nu_{K^*}(y_x)= -Dv(x)/|Dv(x)| =\nu_{\ot{t}}(x)$, being $\ot{t}$ the superlevel set of $v$ of level $t=v(x)$.
\end{proposition}
\begin{proof}
Let $u^*$ be the quasi-concave envelope of $u=\max\{\ua,\ub\}$; it satisfies in the viscosity sense
$$
\begin{cases}
\Deltap u^* \ge 0 &\qquad\text{ in }\oo\setminus K^*\\
u^*=0            &\qquad\text{ on }\bd\oo\\	
u^*=1            &\qquad\text{ on }\bd K^*,
\end{cases}
$$
and hence, by the viscosity comparison principle, 
\begin{equation}\label{vu*}
|Dv|\le |Du^*|\text{ on }\bd K^*.
\end{equation}
Consider $y\in\bd K^*$; then either $y\in \bd K^*\cap\bd K$ or $y\in\bd K^*\setminus \bd K$.

Assume $y\in\bd K^*\cap\bd K$, so that $\nu_{K}(y)=\nu_{K^*}(y)$.
Then either $y\in\bd K_0$, or $y\in\bd K_1$ and hence $|Du^*(y)|\le|D\ua(y)|$ or $|D\ub(y)|$; however in both the cases
$$
|Dv(y)|\le|Du_i(y)|\le g\big(\nu_{K}(y)\big)=g\big(\nu_{K^*}(y)\big),
$$
as $\ua,\ub\in\fm(\oo,g)$.

Now assume $y\in\bd K^*\setminus\bd K$.
By  Proposition 3.1 in \cite{CS}  there exist $x_1,...,x_N\in \bd(K_0\cup K_1)$ such that $x_1,...,x_l\in\bd K_0$, $x_{l+1},...,x_N\in\bd K_1$ (with $0\le l\le N$) and $\la\in\La_N$ such that
$$
\nu_{K_0}(x_i)=\nu_K(x_i)\text{ parallel to }\nu_{K_1}(x_j)=\nu_K(x_j)\text{ parallel to }\nu_{K^*}(y), 
$$ 	
for $i=1,...,l$, $j=l,...,N$ and $y=\sum_{i=1}^N\la_ix_i$. 
Moreover thanks to Proposition 2.2 in \cite{BLS} it holds
$$
|Du^*(y)|=\left( \sum_{k=1}^N \frac{\la_k}{|Du_{i_k}(x_k)|} \right)^{-1}\le \left( \sum_{k=1}^N \frac{\la_k}{g(\nu_{K_{i_k}}(x_k))} \right)^{-1} =\left( \sum_{k=1}^N \frac{\la_k}{g(\nu(x))} \right)=g(\nu(x)),
$$
where $i_k\in\{0,1\}$.
Hence, by (\ref{vu*}), $v\in\fm(\oo,g)$.

Notice that, as $\oo,K^*$ are convex, the function $v$ is quasi-concave, in particular, thanks Lewis's result \cite{L}, $v\in \qm(\oo\setminus K^*)$. 
For every $x\in\oo\setminus K^*$, let $\nu_{\ot{t}}(x)$ be the outer unit normal vector to the superlevel set $\{v(y)\ge v(x)\}$; hence by Lemma 4.1 in \cite{BS}, it holds
$$
|Dv(x)|\le g(\nu_{K^*}(y_x)),
$$
where $y_x\in\bd K^*$ is such that $\nu_{K^*}(y_x)=\nu_{\ot{t}}(x)$. 
\end{proof}

For the sake of completeness we rewrite here two lemmas in \cite{HS4} which are particularly useful in the proof of Theorem \ref{gsub-sol}. 
\begin{lemma}[\cite{HS4}]\label{HS4_211}
Let $D_R=\{x_1<1\}\setminus B_R$, where $B_R=B(x_R,R)$ and $x_R=(-R,0,\dots,0)$. 
Assume $l>0$ and let $u_R$ solve
$$
\begin{cases}
\Deltap u=0 &\qquad\text{ in }D_R\\
u=l       &\qquad\text{ on }\{x_1=1\}\\
u=0        &\qquad\text{ on }\bd B_R.	
\end{cases}
$$
Then for any $\ep>0$ there exists $R$ sufficiently large such that $|Du_R|\le l+\ep$ on $\bd B_R$.
\end{lemma}

\begin{lemma}[\cite{HS4}]\label{lemmablowup}
Let $u$ be the $\p$-capacitary potential of the convex ring $\omk$, with $|Du|\le C$ uniformly in $\omk$.
Then any converging blow-up sequence
$$
u_{r_j}(x)=\frac 1{r_j}\ \left(1-u(r_j x)\right),
$$
at any boundary points gives a linear function $u_0=\alpha x_1^+$, after suitable rotation and translation, where $\alpha=|Du(O)|$ and $O$ indicates the origin.
\end{lemma}
%
%
%

Following the idea of the proof of Theorem 1.2 in \cite{HS4}, now we present the proof of Theorem \ref{gsub-sol}.
\begin{proof}[Proof of Theorem \ref{gsub-sol}.]
Let us consider $u=\sup\{v\ : v\in\fm(\oo,g)\}$, and let $u_n$ be a maximizing sequence.
Notice that, thanks to Proposition \ref{gconvU}, we can assume $\{u_n\}$ to be an increasing sequence of the $\p$-capacitary potentials of convex rings $\oo\setminus\overline{K_n}$, with $|Du_n(x)|\le g(\nu_{K_n}(x))$ on $\bd K_n$ for every $n$.
Let $K$ be the increasing limit of $K_n$; hence $K$ is convex and, as uniform limit of $\p$-harmonic functions, $u$ is the $\p$-capacitary potential of $\omk$, with $|Du(x)|\le g(\nu_K(x))$ on $\bd K$.

We need to show that in fact $|Du(x)|=g(\nu_K(x))$ and we will prove it by contradiction, constructing a function $w\in\fm(\oo,g)$ such that $w\ge u$ with $w>u$ at some point.
Let us remind that $\nu(x)$ indicates the outer unit normal vector to $\bd K$ at $x$.

Let us assume by contradiction that there exists a point $y\in\bd K$ such that 
$$
\alpha=|Du(y)|<g(\nu(y))
$$ 
and assume $y$ to be the origin $O$ with outer unit normal $\nu$ parallel to the first axis. 
Let $\delta$ be such that 
\begin{equation}\label{alpha}
\alpha + 3\delta < g(\nu).
\end{equation}
By Lemma \ref{lemmablowup} the sequence 
$$
\urj=\frac 1{r_j}\left(1-u(r_jx)\right),
$$ 
converges to $\ua(x)=\alpha x_1^+$, hence for every $\eta>0$,
\begin{equation}\label{u>}
u(x)>1-\alpha x_1^+ -\eta r_j,
\end{equation}
if $r_j$ is small enough, for $x=(x_1,...,x_N)\in B(O,r_j)$.

Consider 
$$
w_R(x)=w_{R,\ep}(x)= \left(\alpha+\frac{\delta}2\right) \left( \frac{u_R-\ep}{\alpha+\delta/2-\ep} \right)^+,
$$
where $u_R$ is as in Lemma \ref{HS4_211} and $l=\alpha+\delta/2$.
Then there exist $\ep_0,R_0>0$ such that for $\ep\le\ep_0$ and $R\ge R_0$, 
\begin{equation}\label{DwR}
|Dw_R|\le \alpha+2\delta,\text{ on }\bd\{u_R\le\ep\}=\{w_R=0\}.
\end{equation} 
Moreover there exist $\delta_1,\delta_2>0$ such that 
$$
w_R>\alpha x_1^+ +\delta_2\quad\text{ on }\bd B(O,1)\cap\{x_1>-\delta_1\},
$$
in particular we can fix $\delta_1$ small enough such that $\{u_R=\ep\}\cap\bd B(0,1)\subseteq\{x_1>-2\delta_1\}$, and choose 
$$
0<\delta_2=2\inf\{u_R(x)-\alpha {x_1}^+\ :\ x\in \bd B(O,1)\cap\{x_1>-\delta_1\}\}.
$$
Let $\tilde{w}(x)=1-r_j w_R(x/r_j)$; notice that, as $u_R$ is quasi-convex, then $\tilde w$ is quasi-concave.
Moreover for $r_j$ sufficiently small, recalling (\ref{u>}) it holds
$$
\tilde{w}< 1- \alpha x_1^+ -\delta_2 r_j<u\quad\text{ on }\bd B(O,r_j).
$$

Define
$$
w(x)=
\begin{cases}
\max\{u(x),\tilde{w}(x)\}&\qquad\text{ in }B(O,r_j),\\	
u(x)                     &\qquad\text{ in }\rn\setminus B(O,r_j),
\end{cases}
$$
and $W=\{\tilde{w}=1\}=r_j\{w_R=1\}$; observe that on $\bd B(O,r_j)$, $w=\tilde{w}$. 
By (\ref{alpha}) and (\ref{DwR}), for every $x\in W$ it holds
$$
|D\tilde{w}(x)|\le \alpha +2\delta < g(\nu)-\delta.
$$ 

Notice that $\{u_R=0\}=\bd B(x_R,R)$ and for every $x\in\bd\{u_R=0\}$ it holds 
$$
\lim_{R\to\infty}\nu_{B_R}(x)=\nu=(1,0,...,0).
$$
Moreover $\lim_{\ep\to 0}\{u_R=\ep\}=\{u_R=0\}$ as limit in the \hau metric of {convex sets}.
Hence, by continuity, for sufficiently large $R$ and sufficiently small $\ep$, we have 
$$
|g(\nu)-g(\nu_{W}(z))|\le \delta,
$$ 
for every $z\in W\cap B(O,r_j)$, and hence,
$$
|D\tilde{w}(x)|< g(\nu)-\delta \le g(\nu_W(x)),
$$
for every $x\in W\cap B(o,r_J)$.

Then $w\in\fm(\oo,g)$ and, since $w>u$ at some points, we get a contradiction with the maximality of $u$.
Therefore $|Du|=g(\nu)$ on $\bd K$.
\end{proof}

\section{Final remarks}
\begin{remark}
{\rm  
Notice that in the non constant case no characterization of functions $g$ for which $\fm(\oo,g)$ is not empty are known.
However in some trivial case the existence or non-existence of a solution can be easily deduced by the characterization of the existence for the standard problem in (\ref{Bresume}). 
Indeed if $g$ satisfies 
$$
\min_{\nu\in S^{N-1}} g(\nu)\ge\Lab{\p}(\oo),\qquad\text{ then } \quad\F(\oo,\Lab{\p}(\oo))\subseteq\fm(\oo,g),
$$
and hence $\fm(\oo,g)\neq\emptyset$; on the other hand, if 
$$
M=\max_{\nu\in S^{N-1}} g(\nu)<\Lab{\p}(\oo),\qquad\text{ then }\quad\fm(\oo,g)\subseteq\fm(\oo,M)=\emptyset,
$$ 
and hence problem (\ref{Bint}) has no solutions. 
}
\end{remark}

\begin{remark}[Concavity property of Bernoulli Problems (\ref{Bint})]
{\rm
As in the classical case, geometric properties for the maximal solutions to (\ref{Bint}) can be proved.
Indeed, following the argument in \cite{BS}, it is possible to define a combination of the Bernoulli Problems (\ref{Bint}) in the \mink sense and to prove that Problem (\ref{Bint}) has a concave behaviour with respect to this combination. 
More precisely: fix $\la\in[0,1]$; let $\oo_0,\oo_1$ be two given convex domains  and $g_0,g_1:S^{N-1}\to\R^+$ two continuous functions (which both stay far away from zero). 
We define $\oo_\la$ as the \mink combination of $\oo_0,\oo_1$, that is 
$$
\oo_\la=(1-\la)\oo_0+\la\oo_1=\{z=(1-\la)k_0+\la k_1 \ |\ k_0\in K_0,k_1\in K_1 \},
$$ 
and $g_\la$ as the harmonic mean of $g_0$ and $g_1$, that is 
$$
\frac 1{g_\la(\nu)}=\frac{(1-\la)}{g_0(\nu)}+\frac{\la}{g_1(\nu)}.
$$ 
Consider Problem (\ref{Bint}) for $\oo_0,g_0$ and $\oo_1,g_1$, respectively; we define their \emph{combined problem} of ratio $\la$ the Bernoulli problem of the type (\ref{Bint}), with given set $\oo_\la$ and gradient boundary constraint $g_\la(\nu)$.
Following the proof of Proposition 7.1 in \cite{BS} we can prove that if $\fm(\oo_i,g_i)$, $i=0,1,$ are non empty sets, then  so is $\fm(\oo_\la,g_\la)$.
More precisely let $(\widetilde{K}(\oo_i,g_i),u_i)$ be the maximal solutions, for $i=0,1$ and let $\ula$ be the \mink combination of $u_0$ and $u_1$ of ratio $\la$, that is 
$$
\{\ula\ge t\}=(1-\la)\{u_0\ge t\}+\la\{u_1\ge t\};
$$
(see for instance \cite{BS} for more detailed definitions and properties). 
The function $\ula$ belongs to $\fm(\oo_\la,g_\la)$ and hence, by Theorem \ref{gsub-sol}, Problem (\ref{Bint}) for $\oo_\la$ and $g_\la$ admits a solution $(\widetilde K(\oo_\la,g_\la),\tilde u_\la)$ which satisfies
$$
(1-\la)\widetilde K(\oo_0,g_0(\nu))+\la\widetilde K(\oo_1,g_1(\nu))\subseteq \widetilde K(\oo_\la, g_\la).
$$
}
\end{remark}

\begin{remark}[{A flop in the unbounded case}]
{\rm
It could be natural to try to extend Theorem \ref{gsub-sol} to the unbounded case with an approximation method considering a sequence of given domains $\oo_R=\oo\cap B(O,R)$ as $R$ grows. 
As the sequence $\{\oo_R\}$ is monotone increasing by comparison principle $\tilde{K}(\oo_R,g)$ also increases and hence it converges to a convex set. 
Unfortunately, this approach fails in the limit process as it turns out that in fact $\oo_R$ converges to the given set $\oo$ which means that the limit of maximal solutions degenerates.

More precisely assume for simplicity $\oo=\rn$, so that $\oo_R=B_R=B(O,R)$ (or, analogously $\oo=H^-$ the half space $\{ x_N\le 0 \}$ and take $\oo_R=B_R=B(x_R,R)$, where $x_R=(0,...,0,-R)$).
If $R$ is sufficiently large, then the Bernoulli constant of $B_R$, $\Lab{\p}(B_R)=C_N/R$ (see \cite{BS} for example) is smaller than $c_0$ and hence 
$$
B_r\subseteq \widetilde{K}(B_R,c_0)\subseteq\widetilde{K}(B_R,g(\nu)),
$$
where $B_r=B(O,r)$ is the unique solution to Problem (\ref{Bint}) corresponding to $\oo=B_R$ and $g(\nu)\equiv\Lab{\p}(B_R)$.

Hence for sufficiently large $R$, $\fm(B_R,g)$ is not empty and Theorem \ref{gsub-sol} gives a sequence of quasi-concave $\p$-capacitary potentials $\{u^R\}$ which solve Problem (\ref{Bint}) in $\oo_R\setminus\overline{K_R}$, where $K_R=B_r$. 
By easy computations one can check that $r=R/c_N$, for some constant $c_N$ depending on the dimension and hence, the sequence of interior domains $\{K_R\}_{R>0}$ is not bounded for $R$ which tends to infinity.
This implies that the limit of the maximal solutions $(K_R,u^R)$ is not the solution to the limit problem. 
}
\end{remark}
\subsection*{Acknowledgements}
The author wishes to warmly thank Paolo Salani for his invaluable help, the several helpful discussions and, above all, for his encouragement and support.


\end{document}